\documentclass[12pt,letterpaper,reqno]{amsart}
\usepackage[margin=1in]{geometry}
\usepackage[T1]{fontenc}
\usepackage{lmodern,amssymb,mathtools,microtype,esint}
\usepackage[hidelinks]{hyperref}
\hypersetup{pdftitle={Stability and Derivative Estimates under a Strict Sup-Slope Condition},pdfauthor={Bin Guo and Jian Song},pdfsubject={September 29, 2026: Guo-Phong zeroth-order estimate; no separate oscillation assumption}}
\newtheorem{theorem}{Theorem}[section]
\newtheorem{proposition}[theorem]{Proposition}
\newtheorem{lemma}[theorem]{Lemma}
\newtheorem{corollary}[theorem]{Corollary}
\theoremstyle{remark}
\newtheorem{remark}[theorem]{Remark}
\numberwithin{equation}{section}
\newcommand{\ddbar}{i\partial\bar\partial}
\newcommand{\ub}{\underline u}
\newcommand{\R}{\mathbb R}
\newcommand{\C}{\mathbb C}
\newcommand{\E}{\mathcal E}
\newcommand{\Id}{\mathrm{Id}}
\newcommand{\one}{\mathbf 1}
\newcommand{\Ln}{\mathcal L}
\DeclareMathOperator{\tr}{tr}
\DeclareMathOperator{\osc}{osc}
\DeclareMathOperator{\Vol}{Vol}
\DeclareMathOperator{\Rea}{Re}
\title[Stability and derivative estimates]{Stability and gradient estimates for fully nonlinear elliptic equations on hermitian manifolds}
\author{Bin Guo}
\address{Department of Mathematics \& Computer Science, Rutgers University, Newark, NJ 07102}
\email{bguo@rutgers.edu}
\author{Jian Song}
\address{Department of Mathematics, Rutgers University, Piscataway, NJ 08854}
\email{jiansong@math.rutgers.edu}
\thanks{Work supported in part by the National Science Foundation under grants  DMS-2203607,  DMS-2303508, DMS-2505575 and the Simons Foundation under grant MPS-TSM-00946730.}

\begin{document}

\begin{abstract}
The known proofs of the gradient estimate for complex Hessian equations rely on the second order estimate of Hou-Ma-Wu and on a Liouville theorem of Dinew-Ko\l odziej. In this paper, we give a direct proof of the gradient estimate for concave fully nonlinear elliptic equations on compact Hermitian manifolds, which uses neither the Hou-Ma-Wu estimate nor the Liouville theorem. Instead, the proof combines a stability estimate with a comparison argument against nearby smooth admissible functions. We also give an independent proof of the complex Hessian estimate on compact K\"ahler manifolds, based on Dirichlet Green's functions, under additional assumptions on the operator.

\end{abstract}
\maketitle

\section{Introduction}
The theory of nonlinear elliptic equations on compact complex manifolds has been shaped to a large extent by the complex Monge-Amp\`ere equation and its applications to geometry. Yau's solution of the Calabi conjecture \cite{Yau} established the existence of smooth solutions on compact K\"ahler manifolds for any prescribed positive volume form with the correct total mass. The corresponding result on compact Hermitian manifolds was obtained later by Tosatti-Weinkove \cite{TW}, where the right-hand side has to be rescaled by a constant.

For complex Hessian equations, a central difficulty is the interplay between the gradient and the second order estimates. Hou-Ma-Wu \cite{HMW} obtained a bound for the second derivatives in terms of the gradient. Dinew-Ko\l odziej \cite{DK} then established the gradient estimate by a blow-up argument, together with a Liouville theorem for bounded maximal $m$-subharmonic functions with bounded gradient, thereby completing the theory of smooth solvability on compact K\"ahler manifolds. 
Sz\'ekelyhidi \cite{Szekelyhidi} developed a general framework for fully nonlinear elliptic equations on compact Hermitian manifolds, establishing a priori $C^{2,\alpha}$ estimates under the existence of $\mathcal C$-subsolution. His proof combines a second order estimate in terms of the gradient with a blow-up argument and a Liouville theorem, extending the approach of Hou-Ma-Wu and Dinew-Ko\l odziej to a broad class of operators. As an application, he obtained solvability results for Hessian quotient equations on compact K\"ahler manifolds under a suitable subsolution hypothesis.

Very recently, Chu-Liu \cite{CL} gave a proof of the gradient estimate for complex $k$-Hessian equations on compact K\"ahler manifolds which combines a logarithmic modulus of continuity with the Hou-Ma-Wu estimate and local $W^{2,p}$ estimates. Their approach relies on stability and on auxiliary Monge-Amp\`ere equations, including the construction of Guo-Phong \cite{GP} in the boundary case. It avoids the blow-up analysis and the Liouville theorem of Dinew-Ko\l odziej \cite{DK}, and it gives explicit constants. Mostly important, their proof avoids the blow-up argument and the Liouville theorem. However, since it still depends on the Hou-Ma-Wu estimate, it requires bounds on the second derivatives of the $k$-th root of the right-hand side. 

The purpose of the present paper is to obtain the gradient estimate directly, without recourse to either the second order estimate or the Liouville theorem. We work under the strict sup-slope condition of Guo-Song \cite{GS}, with a fixed smooth admissible subsolution. In this setting, the zeroth order estimate is provided by the work of Guo-Phong on $\mathcal C$-subsolutions \cite[Theorem 3.1]{GPsub}. The main observation of this paper is that it suffices to compare a solution with a single smooth admissible function which is sufficiently close to it in $C^0$. The required distance is determined before the comparison function is chosen, and it does not depend on the derivatives of that function. The compactness needed to make this choice is then provided by stability. Once the gradient is bounded, Sz\'ekelyhidi's estimate gives the second order estimate in the usual way for non-degenerate data.

We now state our results more precisely.

Let $(X,\omega)$ be a  compact Hermitian manifold of complex dimension $n\ge2$, and let $\chi$ be a fixed smooth real $(1,1)$-form. For a real-valued smooth function $w$, we write
\[
\chi_w=\chi+\ddbar w,\qquad A_w=\omega^{-1}\cdot \chi_w,\qquad F(A)=f(\lambda(A)).
\]
Here $\lambda(A)$ denotes the eigenvalues of the Hermitian endomorphism $A$. We assume that $\Gamma$ is an open convex symmetric cone satisfying $\Gamma_n\subset\Gamma\subset\Gamma_1$, and that $f\in C^\infty(\Gamma)\cap C^0(\overline\Gamma)$ is a symmetric function satisfying
\begin{equation}\label{eq:structure}
f_i = \frac{\partial f}{\partial \lambda_i}>0,\quad f\text{ is concave},\quad f>0\text{ in }\Gamma,\quad f=0\text{ on }\partial\Gamma,
\quad \lim_{t\to\infty}f(t\lambda)=\infty\quad(\lambda\in\Gamma).
\end{equation}
A function is said to be {\em admissible} if $\lambda(A_w)\in\Gamma$ everywhere. We denote by $\E$ the set of smooth admissible functions. Note that no positivity condition is imposed on $\chi$ itself.

Following the sup-slope framework of Guo-Song \cite{GS}, we set
\begin{equation}\label{eq:finfty}
f_{\infty,i}(\mu)=\lim_{t\to\infty}f(\mu+te_i),\qquad
f_\infty(\mu)=\min_{1\le i\le n}f_{\infty,i}(\mu),
\end{equation}
where $e_i \in\mathbb R^n$ denote the standard basis vectors of $\mathbb R^n$. 
These limits may be infinite. We assume that there is a fixed smooth admissible function $\ub$ such that
\begin{equation}\label{eq:strict}
f_\infty(\lambda(A_{\ub}(x)))\ge h(x)+\delta_0,\qquad\forall x\in X
\end{equation}
for some $\delta_0>0$ and for every right-hand side under consideration. We note that $\lambda(A_{\ub})\in\Gamma$ is required here, whereas the usual notion of $\mathcal C$-subsolution allows more general functions.

Unless otherwise indicated, all integrals and $L^p$ norms are taken with respect to $dV_\omega=\omega^n$, and the norm of the gradient is defined by $|\nabla w|_\omega^2=g^{i\bar j}w_iw_{\bar j}$. 

Throughout the paper, the fixed data consist of $(X,\omega),\chi,f,\Gamma$, the smooth admissible function $\ub$, and the positive constants $H_0,\delta_0$. We consider right-hand sides $0<h\le H_0$ satisfying \eqref{eq:strict}, and whenever a norm of $u$ is involved, the solutions are normalized by $\inf_Xu=0$. As we shall recall in Section~\ref{sec:continuity}, the strict gap implies a quantitative $\mathcal C$-subsolution condition, and the estimate of Guo-Phong then bounds $\|u\|_{L^\infty}$ in terms of these data alone. In particular, no separate bound on the oscillation needs to be assumed.

A stability estimate, and its application to uniform convergence, were established by Cheng-Xu \cite[Theorem 4.1 and Lemma 4.4]{CX}. Under our stronger assumption that the common subsolution is admissible, a shorter argument based on an auxiliary Monge-Amp\`ere equation yields a modest refinement of their estimate. It is given in Proposition~\ref{prop:stability}, as part of the proof of uniform continuity in Section~\ref{sec:continuity}. Our main result is the following gradient estimate:

\begin{theorem}[Gradient estimate]\label{thm:gradient}
Fix the data above and a constant $H_1>0$. Then every smooth admissible solution of
\begin{equation}\label{eqn:1.1}
F(A_u)=h,\qquad 0<h\le H_0,\qquad \|\nabla h\|_{L^\infty}\le H_1
\end{equation}
for which \eqref{eq:strict} holds satisfies
\begin{equation}\label{eq:gradient}
\|\nabla u\|_{L^\infty}\le C(X,\omega,\chi,f,\Gamma,\ub,H_0,H_1,\delta_0).
\end{equation}
The constant does not depend on $\inf_Xh$, and only the above bound on the first derivatives of $h$ is required.
\end{theorem}

The key ingredient in the proof is the comparison estimate of Proposition~\ref{prop:comparison}, which is stated and proved in Section~\ref{sec:gradient}.

\begin{corollary}[Hessian estimate in the Hermitian case]\label{cor:hermitian}
Fix the data above, together with constants $h_->0$ and $H_2>0$. Then every smooth admissible solution on $(X,\omega)$ satisfying
\[
F(A_u)=h,\quad h_-\le h\le H_0,\quad \|h\|_{C^2(X)}\le H_2,\quad
\inf_Xu=0,
\]
and \eqref{eq:strict} satisfies
\begin{equation}\label{eq:hermitian}
\sup_X|\ddbar u|_\omega\le C,\qquad \|u\|_{C^{2,\alpha_0}(X)}\le C'
\end{equation}
for some $\alpha_0\in(0,1)$, where the constants depend only on the stated data. Here $\omega$ is an arbitrary Hermitian metric and $\chi$ an arbitrary smooth real $(1,1)$-form; no closedness assumption and no further restriction on the cone are needed.
\end{corollary}

Corollary~\ref{cor:hermitian} is a consequence of Theorem~\ref{thm:gradient} and the second order estimate of Sz\'ekelyhidi \cite[Proposition 13]{Szekelyhidi}; the subsolution hypothesis and the dependence of the constants are verified in Section~\ref{sec:hermitian}. A different argument, based on Dirichlet Green's functions, gives the following estimate independently of Theorem~\ref{thm:gradient}, at the cost of additional assumptions on the operator and of a K\"ahler background.

\begin{theorem}[Complex Hessian estimate in the K\"ahler case]\label{thm:kahler}
Assume the structural and common subsolution hypotheses stated above, and assume in addition that $\omega$ is K\"ahler, $\chi$ is closed, and $\Gamma\subset\Gamma_2$. Fix $h_->0$ and $K_2>0$, and suppose that
\[
h_-\le h\le H_0,\qquad \|\log h\|_{C^2(X)}\le K_2.
\]
Assume also that, on the set of levels $\{h_-\le f\le H_0\}$, the operator satisfies the growth condition \eqref{eq:growth} for large trace, the normalized ellipticity condition \eqref{eq:badcone}, and the refined concavity condition \eqref{eq:refined} formulated below. Then
\begin{equation}\label{eq:kahler}
\sup_X|\ddbar u|_\omega\le C.
\end{equation}
The constant depends only on the preceding data and on the constants and moduli in these additional conditions. In particular, no prior gradient estimate is needed. Furthermore, for normalized solutions, there exist $\alpha_0\in(0,1)$ and $C'$ depending only on the data such that $\|u\|_{C^{2,\alpha_0}(X)}\le C'$.
\end{theorem}

We note that both second order results require a positive lower bound for $h$, whereas the stability estimate and the gradient estimate remain uniform as $h$ tends to zero. The additional assumptions needed for the Green's function argument are formulated in Section~\ref{sec:directhessian}; they play no role in Corollary~\ref{cor:hermitian}.

We briefly describe the idea of the comparison argument. As usual, the subsolution $\ub$ provides an alternative for the linearized operator. The new ingredient is a nearby admissible function $v$, which allows us to add a large multiple of $(u-v-\eta)^2$ to the weight in the test function for the gradient. The second derivatives of this term produce a favorable quadratic term, while its first derivatives remain under control, since $u-v$ is small. After differentiating the equation and commuting Chern covariant derivatives, the favorable term absorbs the remaining mixed terms. The resulting estimate depends on the gradient of $v$, but the required distance between $u$ and $v$ does not. Consequently, in a uniformly Cauchy sequence of solutions, a single fixed member of the sequence can serve as $v$ for all the subsequent ones.

The paper is organized as follows. In Section~\ref{sec:continuity}, we establish the uniform continuity and compactness of the solutions, following the stability argument of Cheng-Xu and making use of an auxiliary Monge-Amp\`ere equation on Hermitian manifolds. The analytic inputs for this equation are the existence theorem of Tosatti-Weinkove \cite{TW} and the estimates of Guo-Phong \cite{GP}. Section~\ref{sec:gradient} is devoted to the gradient estimate and the proof of Theorem~\ref{thm:gradient}. The normalization by the sup-slope is discussed in Section~\ref{sec:normalization}, and Corollary~\ref{cor:hermitian} is proved in Section~\ref{sec:hermitian}. Finally, the direct proof of Theorem~\ref{thm:kahler} by Green's functions is given in Section~\ref{sec:directhessian}.

\subsection*{Acknowledgements} Part of this work was carried out with the assistance of ChatGPT.  Proposition~\ref{prop:comparison}  and its proof were first suggested by ChatGPT in the setting of complex Hessian equations; we find it elegant and, to our knowledge, completely new. AI assistance was also used in developing the alternative proof of the complex Hessian estimates in Section~\ref{sec:directhessian}.

\section{Uniform continuity of the solutions}\label{sec:continuity}
We begin by establishing the compactness needed for the gradient estimate. The argument follows Cheng-Xu \cite{CX}: admissibility gives compactness in $L^1$, and the stability estimate upgrades it to uniform convergence.

\subsection{Consequences of the strict sup-slope condition}
Let $\mu(x)=\lambda(A_{\ub}(x))$, with the eigenvalues arranged in increasing order. A first consequence of the strict gap is that positive perturbations of the subsolution are bounded, as long as the value of the operator stays below a fixed level.

\begin{lemma}\label{lem:pinching}
Let $H$ be a fixed positive continuous function, and suppose that
\[
f_\infty(\mu(x))\ge H(x)+4\delta\qquad(x\in X)
\]
for some $\delta>0$. Then there exist $r>0$ and $R<\infty$ such that
\begin{equation}\label{eq:margin}
\mu(x)-2r\one\in\Gamma,\qquad
f_\infty(\mu(x)-2r\one)\ge H(x)+3\delta
\end{equation}
everywhere. Moreover, for any positive semi-definite Hermitian endomorphism $P$ at $x$,
\begin{equation}\label{eq:pinching}
F(A_{\ub}(x)-r\Id+P)\le H(x)+2\delta
\quad\Longrightarrow\quad 0\le P\le R\Id.
\end{equation}
\end{lemma}
\begin{proof}
Each $f_{\infty,i}$ is an increasing limit of positive concave functions on $\Gamma$. We claim that it is either finite everywhere, in which case it is concave and continuous, or infinite everywhere. Indeed, if it were infinite at one point and finite at an interior point, then extending the segment joining them slightly beyond the latter point and applying concavity would lead to a contradiction. By symmetry, the same alternative holds for every $i$. Thus $f_\infty$ is either finite and continuous, or identically infinite. The compactness of $\mu(X)\Subset\Gamma$ now implies \eqref{eq:margin}, after decreasing $r$ if necessary.

Suppose now that \eqref{eq:pinching} fails. Then there exist points $x_j\to x_*$ and positive semi-definite endomorphisms $P_j$ with $\|P_j\|\to\infty$ such that
\[
F(A_{\ub}(x_j)-r\Id+P_j)\le H(x_j)+2\delta.
\]
Let $\lambda^{(j)}$ denote the ordered eigenvalues of $A_{\ub}(x_j)-r\Id+P_j$. By Weyl's monotonicity principle,
\[
\lambda^{(j)}_i\ge\mu_i(x_j)-r\qquad(1\le i\le n).
\]
Since $\sum_i\lambda^{(j)}_i=\sum_i\mu_i(x_j)-nr+\tr P_j\to\infty$, some component, say $\lambda^{(j)}_k$, tends to infinity along a subsequence. For each fixed $t>0$, the continuity of $\mu$ then implies that, for all $j$ sufficiently large,
\[
\lambda^{(j)}\ge\mu(x_*)-2r\one+te_k
\]
componentwise. Using the ellipticity of the operator and letting first $j\to\infty$ and then $t\to\infty$, we obtain
\[
\liminf_jF(A_{\ub}(x_j)-r\Id+P_j)
\ge f_{\infty,k}(\mu(x_*)-2r\one)
\ge H(x_*)+3\delta.
\]
In view of the continuity of $H$, this contradicts the upper bound above. Note that all the matrices used in this comparison are admissible, since $\Gamma+\Gamma_n\subset\Gamma$ by the convexity of $\Gamma$.
\end{proof}

\begin{remark}\label{rem:envelope}
Lemma~\ref{lem:pinching} will be applied with a single upper bound for all the right-hand sides. If $f_\infty$ is finite, we set
\[
H(x)=\min\{H_0,f_\infty(\mu(x))-\delta_0/2\};
\]
if $f_\infty\equiv\infty$, we set $H=H_0$ instead. For any nonempty family satisfying \eqref{eq:strict}, this function is positive, and we have $h\le H$ and $f_\infty(\mu)\ge H+\delta_0/2$. We then take $\delta=\delta_0/8$. The constants $r,R,\delta$ depend only on the prescribed data, including the common gap $\delta_0$; in particular, they do not depend on the modulus of continuity of any individual $h$.
\end{remark}

\begin{lemma}\label{lem:quantC}
Assume the structural and common strict sup-slope hypotheses of the introduction, and let $\mu(x)=\lambda(A_{\ub}(x))$. Then there exist $d>0$ and $R_{\mathrm{sub}}<\infty$, depending only on $f,\Gamma,A_{\ub},H_0,\delta_0$, such that $\mu(x)-2d\one\in\Gamma$ and
\begin{equation}\label{eq:quantC}
\bigl(\mu(x)-2d\one+\Gamma_n\bigr)
\cap\{\lambda\in\Gamma:f(\lambda)=h(x)\}
\subset B_{R_{\mathrm{sub}}}(0)
\end{equation}
for every $x\in X$ and every right-hand side in the family. In particular, these constants do not depend on any bound for the derivatives of $h$ or $u$.
\end{lemma}
\begin{proof}
Let $H$ be the function defined in Remark~\ref{rem:envelope}. It is positive and continuous, and satisfies
\[
h\le H,\qquad f_\infty(\mu)\ge H+\delta_0/2.
\]
Applying Lemma~\ref{lem:pinching} with $\delta=\delta_0/8$, we obtain constants $r,R>0$ which do not depend on the individual right-hand side. We set $2d=r$. If $\lambda$ belongs to the intersection in \eqref{eq:quantC}, we can write
\[
\lambda=\mu(x)-r\one+v,\qquad v\in\Gamma_n.
\]
In an eigenbasis for $A_{\ub}(x)$, the positive diagonal endomorphism $P=\operatorname{diag}(v_1,\ldots,v_n)$ satisfies
\[
F(A_{\ub}(x)-r\Id+P)=f(\lambda)=h(x)\le H(x).
\]
It follows from \eqref{eq:pinching} that $P\le R\Id$. Hence $0<v_i\le R$, and
\[
|\lambda|\le\sup_X|\mu|+\sqrt n\,(r+R).
\]
We may therefore take
\begin{equation}\label{eq:radius}
R_{\mathrm{sub}}=1+\sup_X|\mu|+\sqrt n\,(r+R).
\end{equation}
Finally, the inclusion $\mu-r\one\in\Gamma$ is a consequence of \eqref{eq:margin}.
\end{proof}

Since the lemma also bounds the intersections $(\mu(x)+\Gamma_n)\cap\{f=h(x)\}$, the subsolution is in particular a $\mathcal C$-subsolution. The shift by $2d\one$ provides the uniform margin required in the algebraic alternative below. Its existence relies on the admissibility of the fixed subsolution, which is a stronger condition than that of being a $\mathcal C$-subsolution.

\subsection{The $L^\infty$ estimate}
We recall next the consequence of \cite[Theorem 3.1]{GPsub} which will be used throughout the paper.

\begin{lemma}[Guo-Phong]\label{lem:zeroth}
Assume the structural and common strict sup-slope hypotheses of the introduction. Then every smooth admissible solution of $F(A_u)=h$, with $0<h\le H_0$ and $\inf_Xu=0$, satisfies
\begin{equation}\label{eq:zeroth}
0\le u\le C_{\mathrm{osc}},\qquad
C_{\mathrm{osc}}=C(X,\omega,\chi,f,\Gamma,\ub,H_0,\delta_0).
\end{equation}
The constant does not depend on $\inf_Xh$, nor on any norm of the derivatives of $h$ and $u$.
\end{lemma}
\begin{proof}
Set $\widehat\chi=\chi+\ddbar\ub$ and $w=u-\ub$. By Lemma~\ref{lem:quantC}, $\widehat\chi$ is a $(2d,R_{\mathrm{sub}})$-subsolution in the terminology of \cite{GPsub}. Since this form is fixed and smooth, we can choose $\kappa_1,\kappa_2>0$ such that
\[
-n\kappa_1\omega\le\widehat\chi\le\kappa_2\omega.
\]
Theorem 3.1 of \cite{GPsub}, applied with $\theta=-\ub$ in the notation there, gives
\[
\osc_Xw\le C(n,\omega,2d,R_{\mathrm{sub}},\kappa_1,\kappa_2).
\]
Adding $\osc_X\ub$ and using the normalization $\inf_Xu=0$, we obtain \eqref{eq:zeroth}.

We note that the proof in the Hermitian case in \cite[Section 3]{GPsub} uses only these pointwise bounds on the form and the quantitative subsolution condition. It uses neither the closedness of the background form nor any normalization of $h$, and it therefore applies to the smooth real $(1,1)$-form $\chi$ considered here. Moreover, neither the auxiliary local Monge-Amp\`ere equation used there nor the $L^1$ bound resulting from $\Gamma\subset\Gamma_1$ requires any gradient or second order estimate for $u$.
\end{proof}

\subsection{The auxiliary Monge-Amp\`ere equation}

We shall compare the solutions with the solution of a Monge-Amp\`ere equation on the Hermitian manifold, whose density is concentrated on a small set. Since the total volume is not a cohomological invariant in this setting, we need bounds for the normalizing constant as well as for the potential.

Let $\alpha$ be a fixed Hermitian metric, and let $\alpha_\psi=\alpha+\ddbar\psi>0$. It follows from the estimates of Guo-Phong \cite{GP} that, for $p>n$, an upper bound for
\begin{equation}\label{eq:entropy}
\int_X\frac{\alpha_\psi^n}{\alpha^n}
\left|\log\frac{\alpha_\psi^n}{\alpha^n}\right|^p\alpha^n
\end{equation}
implies a positive lower bound for $\int_X\alpha_\psi^n$ and, with the normalization $\sup_X\psi=0$, a bound for $\|\psi\|_{L^\infty}$. Since an $L^q$ bound on the density, for any $q>1$, suffices to control the integral \eqref{eq:entropy}, we obtain the following lemma.

\begin{lemma}\label{lem:auxiliary}
Fix a Hermitian metric $\alpha$, a number $q>1$, and constants $a_0,L>0$. If $\rho$ is a smooth positive function satisfying
\[
\rho\ge a_0,\qquad \fint_X\rho\alpha^n=1,\qquad \fint_X\rho^q\alpha^n\le L,
\]
then the smooth solution of the Monge-Amp\`ere equation
\begin{equation}\label{eq:auxiliary}
(\alpha+\ddbar\psi)^n=c\rho\alpha^n,\qquad
\alpha+\ddbar\psi>0,\qquad \sup_X\psi=0
\end{equation}
satisfies
\begin{equation}\label{eq:auxbounds}
c_0\le c\le a_0^{-1},\qquad -C_\psi\le\psi\le0,
\end{equation}
where $c_0>0$ and $C_\psi$ depend only on $\alpha,q,a_0,L$. Here $\fint$ denotes the average with respect to the indicated volume form.
\end{lemma}
\begin{proof}
The existence of a smooth solution $(\psi,c)$ for a smooth positive $\rho$ is the theorem of Tosatti-Weinkove \cite{TW}. At a maximum point of $\psi$, we have $\alpha+\ddbar\psi\le\alpha$, hence $c\rho\le1$ there, and $c\le a_0^{-1}$. Consequently, the density $\alpha_\psi^n/\alpha^n=c\rho$ satisfies a uniform $L^q$ bound, and the preceding observation, based on the estimates of \cite{GP}, gives $\|\psi\|_{L^\infty}\le C_\psi$ and $\int_X\alpha_\psi^n\ge v_0>0$. Integrating \eqref{eq:auxiliary}, we find $c\int_X\alpha^n=\int_X\alpha_\psi^n$, whence $c\ge v_0/\int_X\alpha^n$.
\end{proof}
\begin{remark}
In the application below, the auxiliary density $\rho$ will be bounded below by $1/2$, even when the right-hand side $h$ of the original equation has no uniform positive lower bound. On a Hermitian manifold, the constant $c$ need not be equal to one, which is why both bounds in \eqref{eq:auxbounds} are needed.
\end{remark}

\subsection{Compactness in \texorpdfstring{$L^1$}{L1}}
The $L^\infty$ estimate, together with admissibility, gives a first compactness result.

\begin{lemma}\label{lem:L1compact}
The family of normalized solutions in Lemma~\ref{lem:zeroth} is precompact in $L^1(X)$.
\end{lemma}
\begin{proof}
Since $\Gamma\subset\Gamma_1$, admissibility implies $\Delta_\omega u\ge-\tr_\omega\chi\ge-C$. Recall that there exists a smooth positive measure $d\mu$ which is invariant under the complex Laplacian of the background metric, in the sense that $\int_X\Delta_\omega w\,d\mu=0$ for every smooth $w$. It can be obtained from the Gauduchon metric in the conformal class of $\omega$, and its density is a fixed function, uniformly comparable to $\omega^n$. By the product rule,
\[
0=\int_X\Delta_\omega(u^2)\,d\mu
=2\int_Xu\Delta_\omega u\,d\mu+2\int_X|\nabla u|_\omega^2\,d\mu.
\]
By Lemma~\ref{lem:zeroth}, $0\le u\le C_{\mathrm{osc}}$, so the last integral is bounded by $C C_{\mathrm{osc}}\mu(X)$. Since the $L^2$ norm of $u$ is also bounded, the asserted $L^1$ precompactness follows from Rellich's theorem.
\end{proof}

\subsection{Stability}
We now establish the estimate which upgrades this compactness to uniform convergence. The proof makes use of the auxiliary Monge-Amp\`ere equation above, with a density concentrated near a maximum point of the difference of the two solutions.
\begin{proposition}\label{prop:stability}
Fix $(X,\omega),\chi,f,\Gamma,\ub$ and positive numbers $H_0,\delta_0$. Let $u_1,u_2\in\E$ be solutions of
\[
F(A_{u_i})=h_i,\qquad 0<h_i\le H_0,\qquad \inf_Xu_i=0.
\]
Assume that \eqref{eq:strict} holds for both $h_1$ and $h_2$, with the same $\ub$ and $\delta_0$. Then there exists a constant, depending only on these fixed data, such that
\begin{equation}\label{eq:stability}
\|u_1-u_2\|_{L^\infty}\le C\bigl(\|h_1-h_2\|_{L^\infty}+\|u_1-u_2\|_{L^1}\bigr).
\end{equation}
The constant $C$ depends neither on bounds for the derivatives of $h_i$, nor on a positive lower bound for $h_i$.
\end{proposition}
This type of stability estimates is first obtained for the complex Hessian equations by Dinew-Ko\l{}odziej \cite{DK0} using pluripotential theory. 
In \cite[Theorem 4.1]{CX}, Cheng-Xu obtained an estimate with right-hand side $\|h_1-h_2\|_{L^\infty}+\|u_1-u_2\|_{L^1}^{\alpha}$ for $0<\alpha<1/(n+1)$, with constants depending on $\|\log h_i\|_{L^\infty}$. In their setting, the common strict subsolution is only required to lie in $\Gamma_\infty$, whereas ours is $\Gamma$-admissible. Under this stronger hypothesis, the uniform zeroth order estimate is provided by Guo-Phong \cite{GPsub}, and with the stated normalization, the $L^1$ term in \eqref{eq:stability} becomes linear, with a constant controlled by $H_0,\delta_0$ and the fixed data of the geometry, the operator, and the subsolution, but independent of $\inf_X h_i$. Both estimates contain the same linear $L^\infty$ term in the difference of the densities, and neither requires bounds on the derivatives of $h_i$. Thus Proposition~\ref{prop:stability} may be viewed as a modest refinement of the stability result of Cheng-Xu.

\begin{proof}
Since the two solutions can be interchanged, it suffices to bound $\max_X(u_2-u_1)$. We first absorb the subsolution into the background form by setting
\[
\theta=\chi+\ddbar\ub,\qquad u=u_1-\ub,\qquad v=u_2-\ub.
\]
Thus $u-v=u_1-u_2$, while $\theta+\ddbar u=\chi_{u_1}$ and $\theta+\ddbar v=\chi_{u_2}$. By Lemma~\ref{lem:zeroth}, we can choose a fixed constant $C_0\ge1$ such that $|u|,|v|\le C_0$. Suppose that
\[
s_0=\max_X(v-u)>0.
\]
We shall show that
\[
s_0\le C\left(\|h_1-h_2\|_{L^\infty}
+\fint_X(v-u)_+\,dV_\omega\right).
\]

Let $r,R,\delta$ be the constants of Lemma~\ref{lem:pinching} and Remark~\ref{rem:envelope}, and set $\alpha=r\omega$. Applying Lemma~\ref{lem:auxiliary} to densities satisfying
\begin{equation}\label{eq:fixedclass}
\rho\ge\tfrac12,\qquad \fint_X\rho\alpha^n=1,
\qquad \fint_X\rho^2\alpha^n\le4,
\end{equation}
we obtain constants $c_0,C_\psi$, which will be fixed for the rest of the proof. Next, we choose $\kappa>0$ such that
\begin{equation}\label{eq:kappachoice}
2C_0\kappa\le\tfrac14,\qquad
\kappa C_0\le\tfrac1{16},\qquad
\tfrac12\kappa C_\psi\le\tfrac1{16}.
\end{equation}
We then choose $M\ge1$ large enough so that
\begin{equation}\label{eq:Mchoice}
2^{-n}c_0Mr^n>R^n,
\end{equation}
and finally $\varepsilon_0>0$ small enough so that
\begin{equation}\label{eq:epschoice}
\frac{16}{7}\varepsilon_0\le
\min\left\{\frac1{2M},\frac1{M^2}\right\}.
\end{equation}
All these constants depend only on the prescribed data. We stress that the bounds for the auxiliary equation are fixed before $M$ is chosen.

The desired estimate holds unless both of the following inequalities are satisfied:
\begin{equation}\label{eq:smallratios}
\|h_1-h_2\|_{L^\infty}\le\kappa\delta s_0,
\qquad
\fint_X(v-u)_+\,dV_\omega\le\varepsilon_0s_0.
\end{equation}
We assume these inequalities and derive a contradiction. Set $t=\kappa s_0$. Since $s_0\le2C_0$, we have $0<t\le1/4$. Consider the superlevel set
\[
\Omega_{s_0/2}=\{x\in X:(1-t)v(x)-u(x)>s_0/2\}.
\]
On this set, we have
\[
v-u=(1-t)v-u+tv>\frac{s_0}{2}-tC_0\ge\frac{7s_0}{16}.
\]
It follows that
\begin{equation}\label{eq:volE}
\frac{\Vol_\omega(\Omega_{s_0/2})}{\Vol_\omega(X)}
\le\frac{16}{7s_0}\fint_X(v-u)_+\,dV_\omega
\le\frac{16}{7}\varepsilon_0.
\end{equation}

Fix a smooth function $\zeta:\R\to[0,1]$ which vanishes on $(-\infty,1/2]$ and is equal to one on $[3/4,\infty)$, and define
\[
\rho=b+M\zeta\left(\frac{(1-t)v-u}{s_0}\right),
\qquad
b=1-M\fint_X\zeta\left(\frac{(1-t)v-u}{s_0}\right)dV_\omega.
\]
The constant $b$ is chosen so that the average of $\rho$ is equal to $1$. Since the cutoff function is supported in $\Omega_{s_0/2}$, it follows from \eqref{eq:epschoice} and \eqref{eq:volE} that
\[
\frac12\le b\le1,\qquad
\fint_X\rho^2\alpha^n
\le2+2M^2\frac{\Vol_\omega(\Omega_{s_0/2})}{\Vol_\omega(X)}\le4.
\]
Here we have used the fact that the averages with respect to $\alpha^n$ and $\omega^n$ coincide, since $\alpha=r\omega$. Thus $\rho$ satisfies \eqref{eq:fixedclass}. We now solve
\[
(\alpha+\ddbar\psi)^n=c\rho\alpha^n,
\qquad \alpha+\ddbar\psi>0,\qquad \sup_X\psi=0.
\]
By Lemma~\ref{lem:auxiliary}, we have $c\ge c_0$ and $-C_\psi\le\psi\le0$.

We apply the maximum principle to the function
\[
\Phi=(1-t)v-u+\frac{t}{2}\psi.
\]
If $x_*$ is a point where $v-u=s_0$, then
\[
\Phi(x_*)\ge s_0-tC_0-\frac{t}{2}C_\psi\ge\frac{7s_0}{8}.
\]
Therefore, at a maximum point $x_0$ of $\Phi$, the inequality $\psi\le0$ implies
\[
(1-t)v(x_0)-u(x_0)\ge\frac{7s_0}{8}.
\]
In particular, the cutoff function is equal to one at $x_0$, and $\rho(x_0)\ge M$.

At this point, the second derivative test gives, for the $(1,1)$-form
\[
P=\alpha+\ddbar\left(\frac{u-(1-t)v}{t}\right),
\]
the inequality
\begin{equation}\label{eq:positiveP}
P\ge\tfrac12(\alpha+\ddbar\psi)+\tfrac12\alpha
\ge\tfrac12(\alpha+\ddbar\psi)>0.
\end{equation}
Next, we observe that
\begin{equation}\label{eq:matrixidentity}
\theta+\ddbar u
=t\bigl[(\theta-\alpha)+P\bigr]+(1-t)(\theta+\ddbar v).
\end{equation}
Both forms on the right-hand side are admissible: indeed, $\theta-\alpha$ is admissible by the choice of $r$, and admissibility is preserved by adding $P>0$. By concavity and \eqref{eq:smallratios}, we then obtain
\begin{equation}\label{eq:concavecomparison}
\begin{aligned}
F\bigl(\omega^{-1}[(\theta-\alpha)+P]\bigr)
&\le h_2+\frac{h_1-h_2}{t}\\
&\le H(x_0)+\frac{\|h_1-h_2\|_{L^\infty}}{t}
\le H(x_0)+\delta.
\end{aligned}
\end{equation}
Lemma~\ref{lem:pinching} then implies $P\le R\omega$ at $x_0$. On the other hand, the auxiliary equation and \eqref{eq:positiveP} give
\[
\frac{P^n}{\omega^n}(x_0)
\ge2^{-n}c\rho(x_0)r^n
\ge2^{-n}c_0Mr^n>R^n,
\]
which is a contradiction. Hence at least one of the inequalities in \eqref{eq:smallratios} must fail, which proves the desired one-sided estimate. Interchanging the roles of $u_1,u_2$ and absorbing the fixed volume factor, we obtain \eqref{eq:stability}.
\end{proof}

We note that the linear dependence on the $L^1$ norm of the difference is a consequence of keeping $M$ fixed. The construction only requires that the volume of the superlevel set, which is bounded by $s_0^{-1}\fint_X(v-u)_+\,dV_\omega$, be small; the strict sup-slope condition then provides the matrix bound which contradicts the auxiliary Monge-Amp\`ere equation.

\subsection{Uniform continuity from stability}
We can now apply Proposition~\ref{prop:stability} to establish the uniform continuity of normalized admissible solutions.

\begin{lemma}\label{lem:C0compact}
Fix the structural and common subsolution data above. If the right-hand sides form an equicontinuous family with $0<h\le H_0$, then the normalized solutions form a precompact family in $C^0(X)$, and they admit a common modulus of continuity, which depends only on the prescribed data and on a common modulus of continuity for $h$. In particular, this applies to the family in Theorem~\ref{thm:gradient}, even when the right-hand sides tend to zero.
\end{lemma}
\begin{proof}
Let $(u_j,h_j)$ be any sequence. Since the $h_j$ are uniformly bounded and equicontinuous, the Arzel\`a-Ascoli theorem gives a subsequence along which $h_j$ converges uniformly. By Lemma~\ref{lem:L1compact}, we may pass to a further subsequence along which $u_j$ is Cauchy in $L^1$. Proposition~\ref{prop:stability} then implies
\[
\|u_j-u_k\|_{L^\infty}
\le C\bigl(\|h_j-h_k\|_{L^\infty}+\|u_j-u_k\|_{L^1}\bigr)\longrightarrow0.
\]
The constant is uniform, since the same $\ub$ and the same strict gap apply to every pair. This proves the compactness in $C^0$.

The common modulus of continuity is a consequence of precompactness. Indeed, given $\varepsilon>0$, choose a finite $\varepsilon/3$-net for the family of solutions in $C^0(X)$. Since this net consists of finitely many continuous functions, there is a common radius $\rho>0$ on which their oscillations are at most $\varepsilon/3$. By the triangle inequality, $|u(x)-u(y)|\le\varepsilon$ for every solution whenever $d_\omega(x,y)<\rho$. Applying this to the full family determined by the fixed data, we obtain a common modulus of continuity depending only on these data. This completes the proof of the compactness needed below, which relies only on the zeroth order estimate of Guo-Phong, admissibility, and stability.
\end{proof}

\section{The gradient estimate}\label{sec:gradient}
The key step is the following comparison estimate, in which the smooth admissible function $v$ is allowed to vary, while $\ub$ remains the prescribed subsolution. At the end of this section, the compactness established in Section~\ref{sec:continuity} will be used to choose $v$, and Theorem~\ref{thm:gradient} will follow.

\begin{proposition}[Comparison estimate]\label{prop:comparison}
Under the hypotheses of Theorem~\ref{thm:gradient}, there exist $1> \eta>0$ and $C>0$, depending only on the data in that theorem, such that, for every smooth $v\in\E$ with $\|u-v\|_{L^\infty}\le\eta$, we have
\begin{equation}\label{eq:comparison}
\|\nabla u\|_{L^\infty}\le C\bigl(1+\|\nabla v\|_{L^\infty}^2\bigr).
\end{equation}
The radius $\eta$ does not depend on the derivatives of $v$, nor on the distance of $\lambda(A_v)$ from $\partial\Gamma$. In particular, no bound on the second derivatives of $v$ is required.
\end{proposition}

We begin by recording the inequalities for the linearized operator which follow from the existence of the subsolution, and then choose a weight adapted to the difference $u-v$.

\subsection{The linearized operator}
We return to the original background form $\chi$, so that the equation is $F(A_u)=h$. We denote by $B=DF(A_u)$ the linearization, and by $L$ the associated linear operator. In a unitary frame in which $A_u$ is diagonal, we have
\[
B=DF(A_u),\quad B_i=f_i(\lambda(A_u)),\quad Lw=\sum_iB_iw_{i\bar i},\quad
\mathcal T=\tr B=\sum_iB_i.
\]
For any Hermitian endomorphism $P$, we denote by $\operatorname{tr}(BP)$ the trace of the product, whether or not $P$ is diagonal in the chosen frame. The arguments in this subsection follow closely those of \cite{Szekelyhidi} (see also \cite{bG}).

\begin{lemma}\label{lem:alternative}
There exist constants $\tau_0,\kappa_0>0$, depending only on the fixed structural data, $H_0$, $\ub$, and $\delta_0$, such that
\begin{equation}\label{eq:linearized}
0\le \operatorname{tr}(BA_u)\le h,\qquad \mathcal T\ge\tau_0,\qquad
L(u-v)\le h-F(A_v)\le h\quad(v\in\E).
\end{equation}
Moreover,
\begin{equation}\label{eq:weakbarrier}
L(\ub-u)\ge-h,
\end{equation}
and at each point, at least one of the alternatives
\begin{equation}\label{eq:alternative}
L(\ub-u)\ge\kappa_0\mathcal T,\qquad B\ge\kappa_0\mathcal T\Id
\end{equation}
holds. The constants do not depend on any lower bound for $h$.
\end{lemma}
\begin{proof}
Let $P$ be an admissible Hermitian endomorphism. By the convexity of $\Gamma$, the ray $A_u+tP$ is admissible for $t\ge0$. By positivity and concavity, we must have $\operatorname{tr}(BP)\ge0$, since otherwise the tangent plane bound $F(A_u+tP)\le h+t\operatorname{tr}(BP)$ would eventually become negative. Similarly, concavity at $0\in\partial\Gamma$, together with $F(0)=0$, gives $\operatorname{tr}(BA_u)\le h$. Next, choose $a_*>0$ with $F(a_*\Id)>H_0+1$. Then
\[
F(a_*\Id)\le h+\operatorname{tr}\bigl(B(a_*\Id-A_u)\bigr)\le H_0+a_*\mathcal T,
\]
so that $\mathcal T\ge a_*^{-1}$. Applying the tangent plane inequality at $A_u$ to $A_v$, we find $F(A_v)\le h+\operatorname{tr}\bigl(B(A_v-A_u)\bigr)$, which proves \eqref{eq:linearized}. Finally, \eqref{eq:weakbarrier} follows by applying $\operatorname{tr}(BP)\ge0$ to $P=A_{\ub}$.

In order to establish the alternative uniformly as $h$ tends to zero, we consider first small values of $h$. Since the fixed function $\ub$ is admissible, there exists $s_*>0$ such that $A_{\ub}-s_*\Id$ is admissible everywhere. Hence
\begin{equation}\label{eq:strongbarrier}
L(\ub-u)=\operatorname{tr}(BA_{\ub})-\operatorname{tr}(BA_u)\ge s_*\mathcal T-h.
\end{equation}
Thus, when $h\le s_*\tau_0/2$, the first alternative holds with $\kappa_0=s_*/2$. For levels $h\in[s_*\tau_0/2,H_0]$, Lemma~\ref{lem:pinching} provides the uniform bound on translated level sets which is required in \cite[Proposition 6]{Szekelyhidi}. More precisely, after shrinking $d>0$ if necessary, the sets
\[
(\mu(x)-2d\one+\Gamma_n)\cap\{f=h(x)\}
\]
are uniformly bounded, as can be seen by taking $2d=r$ and using \eqref{eq:pinching}. That proposition then gives \eqref{eq:alternative} outside a fixed ball in the space of eigenvalues, with uniform constants on this compact interval of positive levels. The eigenvalues remaining in that ball lie in a compact subset of $\Gamma$, since $f\ge s_*\tau_0/2$ and $f|_{\partial\Gamma}=0$, and the positivity and continuity of the $f_i$ give $B_i/\mathcal T\ge\kappa_0>0$ there. The lemma follows by taking the minimum of the resulting constants.
\end{proof}
\begin{remark}
Homogeneity of degree one would imply $\operatorname{tr}(BA_u)=h$, but the weaker inequalities above suffice for our purposes. We note also that the subsolution alternative does not provide a lower bound for $\det B$. For example, the operator $f(\lambda_1,\lambda_2)=2\lambda_1\lambda_2/(\lambda_1+\lambda_2)$ on the positive cone satisfies $\det DF(\operatorname{diag}(R,1))=4R^2/(R+1)^4\to0$.
\end{remark}

\subsection{Differentiating the equation}
Let $\nabla$ denote the Chern connection of $\omega$, and write $u_{ki}=\nabla_i u_k$ and $u_{i\bar j}=\partial_i\partial_{\bar j}u$. Note that the pure second derivatives $u_{ki}$ need not be symmetric. At the point under consideration, we choose holomorphic coordinates such that $g_{i\bar j}=\delta_{ij}$ and $A_u$ is diagonal; the first derivatives of $g$ need not vanish at that point. Then
\begin{equation}\label{eq:diagonal}
u_{i\bar j}=\lambda_i\delta_{ij}-\chi_{i\bar j}.
\end{equation}
We begin by estimating the linearized operator applied to the square of the gradient.

\begin{lemma}\label{lem:unweighted}
There exists a constant $C_g>0$, depending only on the Hermitian metric $\omega$ and on $\chi$, such that
\begin{equation}\label{eq:unweighted}
\begin{aligned}
L|\nabla u|_\omega^2\ge{}&\sum_{i,k}B_i|u_{ki}|^2
+\frac14\sum_iB_i\lambda_i^2\\
&-C_g\mathcal T\bigl(1+|\nabla u|_\omega^2\bigr)
-2H_1|\nabla u|_\omega.
\end{aligned}
\end{equation}
\end{lemma}
\begin{proof}
Let $T_{ik}^p=\Gamma_{ik}^p-\Gamma_{ki}^p$ be the torsion of the Chern connection, and recall that $R^p{}_{i\bar j k}=-\partial_{\bar j}\Gamma_{ik}^p$. Differentiating $\nabla_i u_k=\partial_i\partial_k u-\Gamma_{ik}^p u_p$, we obtain the commutation formula
\begin{equation}\label{eq:commutation}
\nabla_{\bar j}\nabla_i u_k
=\nabla_k u_{i\bar j}+R^p{}_{i\bar j k}u_p-T_{ik}^p u_{p\bar j}.
\end{equation}
Differentiating the equation \eqref{eqn:1.1} once gives
\[
\sum_iB_i\nabla_k u_{i\bar i}=h_k-\sum_iB_i\nabla_k\chi_{i\bar i}.
\]
Using the product rule for the squared norm, together with \eqref{eq:commutation} and its conjugate, we find
\begin{equation}\label{eq:bochner}
\begin{aligned}
L|\nabla u|_\omega^2={}&\sum_{i,k}B_i|u_{ki}|^2
+\sum_{i,k}B_i|u_{k\bar i}|^2
+2\Rea\sum_kh_ku_{\bar k}\\
&-2\Rea\sum_{i,k}B_i(\nabla_k\chi_{i\bar i})u_{\bar k}\\
&-2\Rea\sum_{i,k,p}B_iT_{ik}^pu_{p\bar i}u_{\bar k}
+\mathcal E_\omega.
\end{aligned}
\end{equation}
Here the curvature term satisfies $|\mathcal E_\omega|\le C\mathcal T|\nabla u|_\omega^2$. The torsion term involves only the mixed second derivatives, and can be estimated as follows:
\[
C|\nabla u|_\omega\mathcal T^{1/2}
\left(\sum_{i,k}B_i|u_{k\bar i}|^2\right)^{1/2}
\le\frac12\sum_{i,k}B_i|u_{k\bar i}|^2+C\mathcal T|\nabla u|_\omega^2.
\]
The term involving $\nabla\chi$ in \eqref{eq:bochner} is bounded by $C\mathcal T|\nabla u|_\omega$. Finally, \eqref{eq:diagonal} implies
\[
\sum_{i,k}B_i|u_{k\bar i}|^2\ge\frac12\sum_iB_i\lambda_i^2-C\mathcal T.
\]
Combining these inequalities, we obtain \eqref{eq:unweighted}. Note that we have retained the full term involving the pure second derivatives, which will be needed to complete a square once the weight is introduced.
\end{proof}

\begin{lemma}\label{lem:weighted}
For any real smooth function $\xi$, the weighted gradient satisfies
\begin{equation}\label{eq:weighted}
\begin{aligned}
&e^{-\xi}L\bigl(e^\xi(1+|\nabla u|_\omega^2)\bigr)\\
&\quad\ge\sum_{i,k}B_i|u_{ki}+u_k\xi_i|^2
+\frac14\sum_iB_i\lambda_i^2
+\bigl(1+|\nabla u|_\omega^2\bigr)L\xi\\
&\qquad+\sum_iB_i|\xi_i|^2
+2\Rea\sum_iB_i\lambda_iu_i\xi_{\bar i}
-2\Rea\sum_{i,k}B_i\chi_{i\bar k}u_k\xi_{\bar i}\\
&\qquad-C_g\mathcal T\bigl(1+|\nabla u|_\omega^2\bigr)
-2H_1|\nabla u|_\omega.
\end{aligned}
\end{equation}
\end{lemma}
\begin{proof}
By \eqref{eq:diagonal}, we have at the chosen point
\[
\partial_i|\nabla u|_\omega^2
=\sum_ku_{ki}u_{\bar k}+\lambda_i u_i-\sum_k\chi_{i\bar k}u_k.
\]
On the other hand, the product rule gives
\[
\begin{aligned}
&e^{-\xi}L\bigl(e^\xi(1+|\nabla u|_\omega^2)\bigr)\\
&\quad=L|\nabla u|_\omega^2
+\bigl(1+|\nabla u|_\omega^2\bigr)L\xi\\
&\qquad+\bigl(1+|\nabla u|_\omega^2\bigr)\sum_iB_i|\xi_i|^2
+2\Rea\sum_iB_i\bigl(\partial_i|\nabla u|_\omega^2\bigr)\xi_{\bar i}.
\end{aligned}
\]
Substituting the preceding identity, and combining the terms involving $u_{ki}$ with $|\nabla u|_\omega^2\sum_iB_i|\xi_i|^2$, we obtain $\sum_{i,k}B_i|u_{ki}+u_k\xi_i|^2$; the remaining part of the weight term is $\sum_iB_i|\xi_i|^2$. The lemma now follows from Lemma~\ref{lem:unweighted}.
\end{proof}

\subsection{Choice of the weight}
The weight consists of two parts: the first involves the fixed subsolution, and the second the small difference between $u$ and $v$. Their coefficients are kept separate, so that the required distance between $u$ and $v$ can be chosen independently of the derivatives of $v$.
For parameters $a\ge1$, $\beta>1$ to be determined later, assume that $\|u-v\|_{L^\infty}\le\beta^{-1}$, and set
\begin{equation}\label{eq:weight}
\begin{gathered}
w=u-v-\beta^{-1},\qquad \xi=a(\ub-u)+\beta w^2,\\
K_v=1+\|\nabla\ub\|_{L^\infty}+\|\nabla v\|_{L^\infty}.
\end{gathered}
\end{equation}
Thus $-2\beta^{-1}\le w\le0$. For a real function $q$, we write
\[
|\partial q|_B^2:=\sum_iB_i|q_i|^2.
\]
In particular, the second derivatives of $\beta w^2$ produce the positive term $2\beta|\partial(u-v)|_B^2$.

\begin{lemma}\label{lem:mixed}
There exists a constant $C_a>0$, depending on $a$ and the fixed background data, but independent of $\beta$ and of the derivatives of $v$, such that the following inequalities hold:
\begin{equation}\label{eq:mixedlambda}
\begin{aligned}
\left|2\Rea\sum_iB_i\lambda_i u_i\xi_{\bar i}\right|
\le{}&\frac18\sum_iB_i\lambda_i^2+C_aK_v^4\mathcal T\\
&+C_a\bigl(1+|\nabla u|_\omega^2+K_v^2\bigr)
|\partial(u-v)|_B^2,
\end{aligned}
\end{equation}
\begin{equation}\label{eq:mixedchi}
\begin{aligned}
\left|2\Rea\sum_{i,k}B_i\chi_{i\bar k}u_k\xi_{\bar i}\right|
\le{}&\tfrac12\bigl(1+|\nabla u|_\omega^2\bigr)
|\partial(u-v)|_B^2\\
&+C_a\mathcal T
+C_aK_v\bigl(1+|\nabla u|_\omega^2\bigr)^{1/2}\mathcal T.
\end{aligned}
\end{equation}
\end{lemma}
\begin{proof}
We write the derivative of the weight in the form
\begin{equation}\label{eq:weightderivative}
\xi_i=a(\underline u_i-v_i)+(-a+2\beta w)(u_i-v_i),\qquad
|-a+2\beta w|\le a+4.
\end{equation}
The contribution of the second term in \eqref{eq:weightderivative} to the left-hand side of \eqref{eq:mixedlambda} is at most $C_a|\nabla u|_\omega\left(\sum_iB_i\lambda_i^2\right)^{1/2}|\partial(u-v)|_B$. For the first term, we use
\[
\begin{gathered}
\left|\sum_iB_i\lambda_iu_i(\underline u_{\bar i}-v_{\bar i})\right|
\le K_v\left(\sum_iB_i\lambda_i^2\right)^{1/2}
\left(\sum_iB_i|u_i|^2\right)^{1/2},\\
\sum_iB_i|u_i|^2\le2|\partial(u-v)|_B^2+2K_v^2\mathcal T.
\end{gathered}
\]
Applying the Cauchy-Schwarz inequality to each of the three resulting products, in such a way that each contributes at most $\frac1{24}\sum_iB_i\lambda_i^2$, we obtain \eqref{eq:mixedlambda}. To prove \eqref{eq:mixedchi}, we use
\[
\left(\sum_iB_i|\xi_i|^2\right)^{1/2}
\le aK_v\mathcal T^{1/2}+(a+4)|\partial(u-v)|_B,
\]
and
\[
\left(\sum_iB_i\left|\sum_k\chi_{i\bar k}u_k\right|^2\right)^{1/2}
\le C\mathcal T^{1/2}|\nabla u|_\omega.
\]
The product of these two quantities is bounded by $C_aK_v|\nabla u|_\omega\mathcal T+C_a|\nabla u|_\omega\mathcal T^{1/2}|\partial(u-v)|_B$, and by the Cauchy-Schwarz inequality, the second summand is bounded by $\frac12\bigl(1+|\nabla u|_\omega^2\bigr)|\partial(u-v)|_B^2+C_a\mathcal T$, as claimed.
\end{proof}

These inequalities account for the choice of the term $\beta w^2$ in the weight. Its contribution to $L\xi$ contains the large coefficient $2\beta$, while the coefficient $-a+2\beta w$ in \eqref{eq:weightderivative} remains bounded by $a+4$, since $u-v$ is small in $C^0$. Moreover, admissibility gives $L(u-v)\le h$, so that the calculation does not require any bound on the Hessian of $v$.

\subsection{Proof of the comparison estimate}
\begin{proof}[Proof of Proposition~\ref{prop:comparison}]
Let $\tau_0,\kappa_0$ be the constants of Lemma~\ref{lem:alternative}. We first choose $a\ge1$ large enough so that
\begin{equation}\label{eq:achoice}
(a\kappa_0-C_g-2)\tau_0>4H_0+2H_1+1.
\end{equation}
Next, with $C_a$ as in Lemma~\ref{lem:mixed}, we choose $\beta\ge2C_a+2$ and set $\eta=\beta^{-1}$. These choices depend only on the prescribed data. Let now $v$ be any smooth admissible function satisfying $\|u-v\|_{L^\infty}\le\eta$, and let $\xi$ be the weight defined in \eqref{eq:weight}.

Since $w\le0$, Lemma~\ref{lem:alternative} gives
\begin{equation}\label{eq:Lweight}
\begin{aligned}
L\xi&=aL(\ub-u)+2\beta wL(u-v)+2\beta |\partial(u-v)|_B^2\\
&\ge aL(\ub-u)-4H_0+2\beta |\partial(u-v)|_B^2.
\end{aligned}
\end{equation}
Combining this inequality with Lemma~\ref{lem:mixed} and \eqref{eq:weighted}, we find that, whenever
\begin{equation}\label{eq:Wthreshold}
\bigl(1+|\nabla u|_\omega^2\bigr)\ge C'_a(1+K_v^4),
\end{equation}
with $C'_a$ sufficiently large depending on the prescribed data, we have
\begin{equation}\label{eq:mastergradient}
\begin{aligned}
&e^{-\xi}L\bigl(e^\xi(1+|\nabla u|_\omega^2)\bigr)\\
&\quad\ge\frac18\sum_iB_i\lambda_i^2
+\bigl(1+|\nabla u|_\omega^2\bigr)
\Bigl[aL(\ub-u)-(C_g+2)\mathcal T\\
&\hspace{43mm}+\beta|\partial(u-v)|_B^2-(4H_0+2H_1)\Bigr].
\end{aligned}
\end{equation}
Indeed, before absorption, the coefficient of $|\partial(u-v)|_B^2$ is
\[
(2\beta-C_a-\tfrac12)\bigl(1+|\nabla u|_\omega^2\bigr)-C_aK_v^2,
\]
which is at least $\beta \bigl(1+|\nabla u|_\omega^2\bigr)$ under \eqref{eq:Wthreshold}. By increasing $C'_a$ if necessary, we also have $C_aK_v^4\le \bigl(1+|\nabla u|_\omega^2\bigr)$ and $C_a+C_aK_v\bigl(1+|\nabla u|_\omega^2\bigr)^{1/2}\le \bigl(1+|\nabla u|_\omega^2\bigr)$; these inequalities account for the additional term $2\mathcal T$ in \eqref{eq:mastergradient}. Finally, we have used $|\nabla u|_\omega\le \bigl(1+|\nabla u|_\omega^2\bigr)$.

Let $x_0$ be a maximum point of $e^\xi \bigl(1+|\nabla u|_\omega^2\bigr)$, and suppose that \eqref{eq:Wthreshold} holds there. If the first alternative in \eqref{eq:alternative} holds, then by \eqref{eq:achoice} the bracket in \eqref{eq:mastergradient} is strictly positive, which contradicts $L(e^\xi \bigl(1+|\nabla u|_\omega^2\bigr))(x_0)\le0$. If the second alternative holds, then
\[
|\partial(u-v)|_B^2\ge\kappa_0\mathcal T|\nabla(u-v)|_\omega^2.
\]
When $\bigl(1+|\nabla u|_\omega^2\bigr)\ge4(1+K_v^2)$, the triangle inequality gives
\[
|\nabla(u-v)|_\omega^2\ge c_*\bigl(1+|\nabla u|_\omega^2\bigr),\qquad c_*=(\sqrt3-1)^2/4.
\]
Using \eqref{eq:weakbarrier}, we conclude that the bracket in \eqref{eq:mastergradient} is at least
\[
[\beta\kappa_0c_*\bigl(1+|\nabla u|_\omega^2\bigr)-C_g-2]\mathcal T-(a+4)H_0-2H_1.
\]
Since $\mathcal T\ge\tau_0$, this quantity is positive for $|\nabla u|_\omega$ sufficiently large, which is again a contradiction; this threshold can be incorporated by increasing the fixed constant in \eqref{eq:Wthreshold}. Hence \eqref{eq:Wthreshold} cannot hold at the maximum point, and we have $$1+|\nabla u(x_0)|_\omega^2\le C(1+K_v^4).$$

The weight function $\xi$ satisfies
\[
\osc_X\xi\le a(C_{\mathrm{osc}}+\osc_X\ub)+4/\beta.
\]
Comparing the values of the test function at its maximum point and at an arbitrary point, we obtain $1+\|\nabla u\|_{L^\infty}^2\le C(1+K_v^4)$, and consequently
\[
\|\nabla u\|_{L^\infty}\le C\bigl(1+\|\nabla v\|_{L^\infty}^2\bigr).
\]
\end{proof}

\subsection{Proof of the gradient theorem}
We now choose the comparison function, using the compactness established in Section~\ref{sec:continuity}.

\begin{proof}[Proof of Theorem~\ref{thm:gradient}]
Suppose that the uniform gradient estimate does not hold. Then there exists a sequence of smooth admissible solutions, with the same prescribed data, such that $\|\nabla u_j\|_{L^\infty}\to\infty$. We normalize the solutions by $\inf_Xu_j=0$, and use Lemma~\ref{lem:C0compact} to pass to a subsequence which is uniformly Cauchy.

Let $\eta>0$ and $C$ be the constants of Proposition~\ref{prop:comparison}. We fix an index $N$ large enough so that
\[
\|u_j-u_N\|_{L^\infty}\le\eta\qquad(j\ge N),
\]
and set $v=u_N$. Being one of the original solutions, this function is smooth and admissible, and the proposition gives
\[
\|\nabla u_j\|_{L^\infty}\le C\bigl(1+\|\nabla u_N\|_{L^\infty}^2\bigr)
\qquad(j\ge N).
\]
The right-hand side is finite and independent of $j$, which contradicts the choice of the sequence.
\end{proof}

\begin{remark}\label{rem:quantifiers}
We stress that the radius $\eta$ is fixed before $u_N$ is chosen. Although there is no a priori control on $\|\nabla u_N\|_{L^\infty}$, it is a finite number which remains fixed as $j\to\infty$, and this suffices for the contradiction. The argument gives a bound which depends only on the prescribed data, but it does not provide an explicit expression for this bound.
\end{remark}

\section{Sup-slope normalization}\label{sec:normalization}
Recall from \cite{GS} that the sup-slope of a smooth real function $G$ is defined by
\[
\sigma_G=\inf_{w\in\E}\max_X e^{-G}F(A_w).
\]
If the equation $F(A_u)=e^{G+c}$ admits a smooth admissible solution, then $\sigma_G=e^c$. Indeed, taking $w=u$ gives $\sigma_G\le e^c$. Conversely, for any $w\in\E$, we have $A_u\le A_w$ at a maximum point of $u-w$, and hence $\max_Xe^{-G}F(A_w)\ge e^c$.

Comparing the two weights in the definition, we obtain
\[
e^{-\|G_1-G_2\|_{L^\infty}}\sigma_{G_2}
\le\sigma_{G_1}
\le e^{\|G_1-G_2\|_{L^\infty}}\sigma_{G_2}.
\]
Thus, whenever both normalized equations are solvable,
\begin{equation}\label{eq:slopecomparison}
|c_1-c_2|\le\|G_1-G_2\|_{L^\infty}.
\end{equation}
If, in addition, the densities are bounded above by $H_0$, the mean value theorem applied to the exponential function gives
\[
\|e^{G_1+c_1}-e^{G_2+c_2}\|_{L^\infty}\le2H_0\|G_1-G_2\|_{L^\infty}.
\]
Consequently, Proposition~\ref{prop:stability} also yields
\begin{equation}\label{eq:Gstability}
\|u_1-u_2\|_{L^\infty}
\le C\bigl(\|G_1-G_2\|_{L^\infty}+\|u_1-u_2\|_{L^1}\bigr)
\end{equation}
under its common subsolution hypotheses. The strict sup-slope condition
\[
e^{-G}f_\infty(\lambda(A_{\ub}))>\sigma_G
\]
is exactly a strict comparison with the level $h=e^{G+c}$ of the solution; for families of equations, we require the quantitative common gap \eqref{eq:strict}.

The class of operators considered here includes the roots of the Hessian operators $f=\sigma_k^{1/k}$ on $\Gamma_k$, and the Hessian quotients $f=(\sigma_k/\sigma_\ell)^{1/(k-\ell)}$, $1\le\ell<k\le n$. Our results apply to these operators whenever a common admissible strict subsolution as above is available, and in each case the zeroth order estimate follows from Lemma~\ref{lem:zeroth}. 


\section{Hessian estimates on Hermitian manifolds}\label{sec:hermitian}
With the gradient estimate in hand, we can apply the second order estimate of Sz\'ekelyhidi. We recall this estimate in the form needed here, and then verify that its constants can be chosen uniformly under our assumptions.

\begin{proposition}[Sz\'ekelyhidi]\label{prop:secondorder}
Assume the structural and common strict sup-slope hypotheses of the introduction, and let $u$ be a solution of $F(A_u)=h$ on a compact Hermitian manifold. Suppose in addition that $h_-\le h\le H_0$ and $\|h\|_{C^2}\le H_2$, where $h_-,H_2>0$. Then
\begin{equation}\label{eq:secondorder}
\sup_X|\ddbar u|_\omega
\le C\bigl(1+\|\nabla u\|_{L^\infty}^2\bigr).
\end{equation}
The constant depends only on the fixed data, $h_-$, and $H_2$, and it is independent of the norm of the gradient on the right-hand side. Here the quantitative subsolution condition \eqref{eq:quantC} and the Guo-Phong estimate \eqref{eq:zeroth} provide the subsolution and zeroth order bounds required in Sz\'ekelyhidi's estimate. Neither the metric nor the background form is required to be closed.
\end{proposition}


\subsection{The subsolution condition}
The quantitative $\mathcal C$-subsolution condition has already been verified in Lemma~\ref{lem:quantC}, and the zeroth order bound required in Sz\'ekelyhidi's estimate is provided by Lemma~\ref{lem:zeroth}.

The remaining hypotheses of Proposition~\ref{prop:secondorder} follow from the assumptions made in the introduction. The cone is open, convex, symmetric, and contains $\Gamma_n$, while $f$ is smooth, symmetric, elliptic, and concave. The radial growth assumption in \eqref{eq:structure} gives growth beyond every finite level. Moreover,
\begin{equation}\label{eq:boundaryseparation}
\sup_{\partial\Gamma}f=0<h_-\le\inf_Xh,
\end{equation}
so that the interval $[h_-,H_0]$ lies in a compact subset of the range $(0,\infty)$ of admissible values. Together with Lemma~\ref{lem:quantC}, this gives a uniform constant in the subsolution alternative of \cite[Proposition 6]{Szekelyhidi}. We note that these structural assumptions are imposed on $f$ and $\Gamma$, while the strict sup-slope condition is used only to verify the subsolution condition.

For later use, we note that the lower bound for the trace of the linearization can also be seen directly. Choose $a_*>0$ such that $f(a_*\one)>H_0+1$. As in Lemma~\ref{lem:alternative}, positivity along admissible rays gives $Df(\lambda)\cdot\lambda\ge0$, and concavity then implies, whenever $f(\lambda)\le H_0$,
\[
f(a_*\one)\le f(\lambda)+a_*\sum_i f_i(\lambda)
-\sum_i f_i(\lambda)\lambda_i
\le H_0+a_*\sum_i f_i(\lambda).
\]
Thus $\sum_i f_i(\lambda)\ge a_*^{-1}$, without any bound on the eigenvalues of $A_u$.

\subsection{Dependence of the second order constant}
In order to keep track of the dependence on the subsolution, we write
\begin{equation}\label{eq:shiftedbackground}
w=u-\ub,\qquad \widehat\chi=\chi+\ddbar\ub,
\qquad \widehat\chi+\ddbar w=\chi+\ddbar u.
\end{equation}
Then the zero function is a quantitative $\mathcal C$-subsolution for the equation satisfied by $w$. After subtracting a constant from $w$, we may assume that $\inf_Xw=0$, and
\[
\osc_Xw\le C_{\mathrm{osc}}+\osc_X\ub.
\]
Proposition~\ref{prop:secondorder} then gives
\begin{equation}\label{eq:Cstar}
\begin{gathered}
\|\ddbar w\|_{L^\infty}
\le C_*\bigl(1+\|\nabla w\|_{L^\infty}^2\bigr),\\
C_*=C_*\bigl(X,\omega,f,\Gamma,\|\widehat\chi\|_{C^2},
h_-,H_0,H_2,d,R_{\mathrm{sub}}\bigr).
\end{gathered}
\end{equation}
Here the fixed geometric data include the bounds in coordinates for the metric and its inverse. Lemma~\ref{lem:zeroth}, applied to $w$, bounds its oscillation in terms of $d,R_{\mathrm{sub}}$ and the fixed shifted background form, while Lemma~\ref{lem:quantC} determines $d,R_{\mathrm{sub}}$ in terms of this background form and $f,\Gamma,H_0,\delta_0$. Thus none of these quantities requires an additional assumption on $u$. In the proof of Sz\'ekelyhidi's estimate, the quantity $1+\|\nabla w\|_{L^\infty}^2$ enters as a parameter in the gradient term of the weight, and the maximum principle bounds the ratio of the largest eigenvalue to this parameter. The remaining constants are chosen in terms of the data displayed in \eqref{eq:Cstar}.

Returning to $u$ and using the triangle inequality, we obtain
\begin{equation}\label{eq:unshiftedconstant}
\begin{aligned}
\|\ddbar u\|_{L^\infty}
&\le C_*\bigl(1+2\|\nabla u\|_{L^\infty}^2
+2\|\nabla\ub\|_{L^\infty}^2\bigr)
+\|\ddbar\ub\|_{L^\infty}\\
&\le C_{\mathrm{so}}\bigl(1+\|\nabla u\|_{L^\infty}^2\bigr),
\end{aligned}
\end{equation}
where we may take
\[
C_{\mathrm{so}}=
2C_*\bigl(1+\|\nabla\ub\|_{L^\infty}^2\bigr)
+\|\ddbar\ub\|_{L^\infty}.
\]
Thus the second order constant involves derivatives of the fixed subsolution, but no norm of the derivatives of the unknown solution.

\subsection{Proof of Corollary~\ref{cor:hermitian}}
We note first that the argument is not circular. The zeroth order estimate of Guo-Phong precedes the compactness and gradient arguments, and the gradient estimate was proved using only the algebraic subsolution alternative, the stability estimate, and a maximum principle argument. The integral estimate in Lemma~\ref{lem:L1compact} uses only the Laplacian of the fixed background metric, and in the comparison argument of Section~\ref{sec:gradient}, the terms involving the Hessian are estimated as they arise, without assuming any bound on them. Consequently, Proposition~\ref{prop:secondorder} can now be applied.

\begin{proof}[Proof of Corollary~\ref{cor:hermitian}]
The bound $\|h\|_{C^2}\le H_2$ implies the first derivative bound required in Theorem~\ref{thm:gradient}; with respect to the fixed norms, we may take $H_1=C_\omega H_2$. Hence
\[
\|\nabla u\|_{L^\infty}\le C_{\mathrm{grad}},
\]
where $C_{\mathrm{grad}}$ depends only on $(X,\omega),\chi,f,\Gamma,\ub,H_0,H_2,\delta_0$. By Lemma~\ref{lem:quantC} and the preceding verification, we can apply \eqref{eq:unshiftedconstant}, which yields
\begin{equation}\label{eq:M2}
\|\ddbar u\|_{L^\infty}
\le C_{\mathrm{so}}(1+C_{\mathrm{grad}}^2).
\end{equation}
We note that the subsolution constants $d,R_{\mathrm{sub}}$ and the constant $C_{\mathrm{so}}$ have been chosen without using either $C_{\mathrm{grad}}$ or a bound on the Hessian of $u$.

It remains to pass from the bound on the complex Hessian to the full $C^{2,\alpha}$ estimate. By \eqref{eq:M2}, and since the background form is fixed, we have $|\lambda(A_u)|\le R_1$ for a fixed constant $R_1$. The set
\begin{equation}\label{eq:spectralslab}
\mathcal K=\{\lambda\in\overline\Gamma:
|\lambda|\le R_1,\ h_-\le f(\lambda)\le H_0\}
\end{equation}
is then compact and contained in $\Gamma$, since $f=0$ on $\partial\Gamma$ and $h_->0$. It follows that
\[
0<\vartheta\le f_i(\lambda)\le\Theta<\infty
\qquad(\lambda\in\mathcal K,\ 1\le i\le n)
\]
for constants determined by the data.

We can therefore apply the local complex Evans-Krylov estimate, in the formulation of \cite[Theorem 1.2]{TWWY}, on a fixed finite cover by coordinate charts. Since the metric and $\chi$ are fixed and $h$ has the required regularity, we obtain a uniform $C^{\alpha_0}$ bound for the complex Hessian, for some $\alpha_0\in(0,1)$.

In particular, $\Delta_\omega u$ is bounded in $C^{\alpha_0}$, and the Schauder estimates \cite{GT} for the fixed metric $\omega$ give
\[
\|u\|_{C^{2,\alpha_0}(X)}
\le C\bigl(\|\Delta_\omega u\|_{C^{\alpha_0}(X)}
+\|u\|_{C^0(X)}\bigr)\le C',
\]
where the last term is controlled by Lemma~\ref{lem:zeroth}. This completes the proof of the corollary.
\end{proof}

\begin{remark}\label{rem:quantitativescope}
If one were to allow a varying family of subsolutions, the second order estimate would require uniform smooth bounds on the subsolutions and common constants in \eqref{eq:quantC}; a strict inequality for each solution, without quantitative control, would not suffice. In our setting, these bounds are provided by the fixed smooth admissible function $\ub$ and the common gap. We note also that the gradient estimate itself relies on the stronger hypothesis that the subsolution is admissible. Finally, we do not claim that the second order constants remain uniform as $h_-\to0$.
\end{remark}

\section{A direct Hessian estimate on K\"ahler manifolds}\label{sec:directhessian}
In this section, we give the proof of Theorem~\ref{thm:kahler} based on Green's functions. Throughout the section, $\omega$ is K\"ahler and $\chi$ is closed. The argument uses the uniform compactness established in Section~\ref{sec:continuity}, together with additional assumptions on $f$, in order to control the derivative term in the differential inequality for the largest eigenvalue.

\subsection{Assumptions on the operator}
We arrange the eigenvalues in decreasing order, $\lambda_1\ge\cdots\ge\lambda_n$, and recall that $\mathcal T(\lambda)=\sum_i f_i(\lambda)$. Since $\Gamma\subset\Gamma_2$, we have
\begin{equation}\label{eq:tracecone}
\sum_i\lambda_i^2<\sigma_1(\lambda)^2,
\qquad |\lambda_i|<\sigma_1(\lambda),
\qquad 0<\lambda_1\le\sigma_1(\lambda)\le n\lambda_1.
\end{equation}
Indeed, $\sigma_1(\lambda)^2-\sum_i\lambda_i^2=2\sigma_2(\lambda)>0$. Note also that $\sigma_1(\lambda(A_u))=\tr_\omega\chi_u$ at each point of $X$.

The first additional assumption is that the trace of the linearization grows when the trace of $A_u$ becomes large. With the convention that the supremum over the empty set is zero, we assume that
\begin{equation}\label{eq:growth}
\varepsilon(R):=
\sup_{\substack{\lambda\in\Gamma,\ h_-\le f(\lambda)\le H_0\\
\sigma_1(\lambda)\ge R}}
\frac1{\mathcal T(\lambda)}\longrightarrow0
\qquad(R\to\infty).
\end{equation}
This is an assumption on the operator, independent of the strict sup-slope condition.

We also assume a normalized ellipticity condition in the region where the smallest eigenvalue is sufficiently negative. More precisely, for every $\epsilon>0$, there exist $\vartheta_\epsilon>0$ and $R_\epsilon<\infty$ such that
\begin{equation}\label{eq:badcone}
\begin{gathered}
h_-\le f(\lambda)\le H_0,\qquad
\sigma_1(\lambda)>R_\epsilon,\qquad
\lambda_n\le-\epsilon\sigma_1(\lambda)\\
\Longrightarrow\quad
\frac{f_i(\lambda)}{\mathcal T(\lambda)}\ge\vartheta_\epsilon
\quad(1\le i\le n).
\end{gathered}
\end{equation}
This condition is vacuous if $\Gamma$ is contained in the positive orthant. Otherwise, it provides uniform ellipticity on the set to which the remaining gradient term will be localized.

Finally, we require a refined concavity inequality, which was proved for complex Hessian operators by Dong-Zhang \cite[Lemma 3.2]{DZ}, and which we impose here as an additional hypothesis on a general operator $f$, near the region where the eigenvalues are nonnegative. More precisely, we assume that there exists $C_*>0$ such that, for every $\beta\in(0,1)$, there exist $\delta_\beta\in(0,\beta]$ and $R_\beta<\infty$ with the following property: if
\[
\lambda\in\Gamma,\qquad h_-\le f(\lambda)\le H_0,
\qquad \lambda_1\ge R_\beta,
\qquad \lambda_n\ge-\delta_\beta\lambda_1,
\]
then, for every $\eta\in\C^n$,
\begin{equation}\label{eq:refined}
\begin{aligned}
-\sum_{p,q}f_{pq}\eta_p\overline{\eta_q}
+\sum_{i>1}\frac{f_i|\eta_i|^2}{(1+\delta_\beta)\lambda_1}
\ge{}&(1-\beta)\frac{f_1|\eta_1|^2}{\lambda_1}\\
&-C_*\frac{|\sum_i f_i\eta_i|^2}{f(\lambda)}.
\end{aligned}
\end{equation}
We emphasize that the constant $C_*$ is independent of $\beta$. If the condition holds initially with some unrestricted $\delta_\beta>0$, then we can decrease $\delta_\beta$ so that $\delta_\beta\le\beta$, since this restricts the region and increases the favorable term on the left-hand side. Note that concavity by itself only implies that the first term is nonnegative, and does not imply \eqref{eq:refined}.

For $f=\sigma_k^{1/k}$ on $\Gamma_k$, $2\le k\le n$, inequality~\eqref{eq:refined} follows from \cite[Lemma 3.2, equation (3.2)]{DZ} by the chain rule, with $C_*=k+1$. Indeed, after multiplying their inequality by $\sigma_k$ and differentiating $f=\sigma_k^{1/k}$, the coefficient $2$ of the squared first derivative term in their inequality becomes $k+1$ in our normalization. Thus $C_*$ is independent of $\beta$, and their thresholds can be chosen uniformly on the fixed interval of levels. A version for sums of Hessian operators is also established by Dong-Zhang in \cite[Lemma 3.1]{DZ}.

\subsection{A barrier where the trace is large}
We normalize the linearized operator by its trace:
\[
\Ln=\mathcal T^{-1}L,\qquad
|\partial q|_{\Ln}^2=\mathcal T^{-1}\sum_iB_i|q_i|^2.
\]
For each smooth solution, this is an elliptic operator with smooth positive coefficients, whose trace is equal to one. Outside the region in \eqref{eq:badcone}, no lower bound on the ellipticity which is uniform over the family is assumed.

Choose $s_*>0$ such that $A_{\ub}-s_*\Id$ is admissible everywhere. By \eqref{eq:strongbarrier}, we have
\[
\Ln(u-\ub)\le h/\mathcal T-s_*.
\]
By \eqref{eq:growth}, the first term on the right-hand side is small when $\tr_\omega\chi_u$ is large. Thus, if we set
\begin{equation}\label{eq:highbarrierdef}
w_u=u-\ub-\sup_X(u-\ub),\qquad
C_{\mathrm{shift}}=C_{\mathrm{osc}}+\osc_X\ub,\qquad \kappa=s_*/2,
\end{equation}
there exists a uniform constant $R_0$ such that
\begin{equation}\label{eq:highbarrier}
-C_{\mathrm{shift}}\le w_u\le0,
\qquad \Ln w_u\le-\kappa
\quad\hbox{on }\{\tr_\omega\chi_u>R_0\}.
\end{equation}
Note that only the admissibility of $\ub$ is used here; in particular, $A_{\ub}$ need not be positive definite. The inequalities $\operatorname{tr}(BA_u)\le h$ and $\operatorname{tr}(BA_{\ub})\ge s_*\mathcal T$ also show that homogeneity of the operator is not needed.

The other ingredient which we shall need is the uniform $C^0$ compactness of Lemma~\ref{lem:C0compact}, whose hypothesis on the right-hand sides follows from the bound on $\log h$ in Theorem~\ref{thm:kahler}. We recall that this compactness was established before, and independently of, the gradient estimate.

\subsection{The largest eigenvalue}
\begin{lemma}\label{lem:eigenvalue}
There exists a constant $C$, independent of $\beta$, such that, wherever $\lambda_1$ is sufficiently large,
\begin{equation}\label{eq:globaleigen}
\Ln\log\lambda_1\ge-|\partial\log\lambda_1|_{\Ln}^2-C-C/\lambda_1.
\end{equation}
Moreover, for every $\beta\in(0,1)$, we have on $\{\lambda_1>R'_\beta,\lambda_n\ge-\delta_\beta\lambda_1\}$ the sharper inequality
\begin{equation}\label{eq:refinedeigen}
\Ln\log\lambda_1\ge-\beta|\partial\log\lambda_1|_{\Ln}^2-C-C/\lambda_1.
\end{equation}
At points where $\lambda_1$ is not smooth, these inequalities hold in the sense of smooth test functions touching from above.
\end{lemma}
\begin{proof}
Assume first that $\lambda_1$ is simple. We use K\"ahler normal coordinates at the given point, in which $A_u$ is diagonal, and set
\[
\eta_p=\nabla_1(\chi_u)_{p\bar p},\qquad z_i=\nabla_i(\chi_u)_{1\bar1}.
\]
Since $\chi_u$ is closed, we have $z_i=\nabla_1(\chi_u)_{i\bar1}$ and $z_1=\eta_1$. Differentiating the equation once, we obtain
\begin{equation}\label{eq:diffh}
\sum_p f_p\eta_p=h_1.
\end{equation}
Recall that the Hessian of a spectral matrix function is given by
\begin{equation}\label{eq:spectralhessian}
D^2F(A)[P,P]
=\sum_{p,q}f_{pq}P_{p\bar p}P_{q\bar q}
+\sum_{p\ne q}\frac{f_p-f_q}{\lambda_p-\lambda_q}|P_{p\bar q}|^2.
\end{equation}
Here the quotients are understood as their limiting values at repeated eigenvalues, and they are nonpositive by concavity and symmetry.

Differentiating the equation twice, commuting covariant derivatives, and using the formula for the second derivatives of the largest eigenvalue, we obtain, keeping only the terms needed below,
\begin{equation}\label{eq:Llargest}
\begin{aligned}
L\lambda_1\ge{}&h_{1\bar1}-C\lambda_1\mathcal T
-\sum_{p,q}f_{pq}\eta_p\overline{\eta_q}\\
&+\sum_{i>1}\frac{f_i|\eta_i|^2}{\lambda_1-\lambda_i}
+\sum_{i>1}\frac{f_i|z_i|^2}{\lambda_1-\lambda_i}.
\end{aligned}
\end{equation}
In the last sum, the coefficient $(f_i-f_1)/(\lambda_1-\lambda_i)$ coming from $-D^2F$ combines with the coefficient $f_1/(\lambda_1-\lambda_i)$ coming from the variation of the eigenvalue. The terms involving $\eta_i$ arise from the other eigenvalue variation terms, by the Codazzi identity. The terms arising from the commutation of derivatives are bounded by $C\lambda_1\mathcal T$, in view of \eqref{eq:tracecone}; no sign condition on the curvature is needed.

On the region where $\lambda_n\ge-\delta_\beta\lambda_1$, we have $\lambda_1-\lambda_i\le(1+\delta_\beta)\lambda_1$. We apply \eqref{eq:refined} to the diagonal quadratic form in \eqref{eq:Llargest}, and use \eqref{eq:diffh}. After dividing by $\lambda_1\mathcal T$ and subtracting $|\partial\log\lambda_1|_{\Ln}^2$, the loss in the term with $i=1$ is at most $\beta(f_1/\mathcal T)|z_1|^2/\lambda_1^2$. For $i>1$, the remaining coefficient is
\[
\frac{f_i}{\mathcal T}\left(\frac1{\lambda_1(\lambda_1-\lambda_i)}-\frac1{\lambda_1^2}\right)
=\frac{f_i}{\mathcal T}\frac{\lambda_i}{\lambda_1^2(\lambda_1-\lambda_i)}
\ge-\delta_\beta\frac{f_i/\mathcal T}{\lambda_1^2}.
\]
Since $\delta_\beta\le\beta$, this gives the coefficient $\beta$ in \eqref{eq:refinedeigen}. The remaining term is controlled by the assumed $C^2$ bound on $\log h$:
\[
h_{1\bar1}-C_*\frac{|h_1|^2}{h}
=h\bigl((\log h)_{1\bar1}+(1-C_*)|(\log h)_1|^2\bigr)\ge-C.
\]
Since $\mathcal T\ge\tau_0$ by Lemma~\ref{lem:alternative}, the corresponding error term, after normalization, is at most $C/\lambda_1$. This proves \eqref{eq:refinedeigen}. To obtain \eqref{eq:globaleigen}, it suffices to discard all the favorable quadratic terms in the twice differentiated equation, and to subtract the full term involving the derivative of the logarithm.

It remains to explain how these inequalities are to be understood when $\lambda_1$ is not differentiable. Suppose that a smooth function touches $\lambda_1$ from above at a point where the top eigenspace has dimension $m>1$. The first order contact conditions force each directional derivative of the top block of $A_u$ to be a scalar multiple of the identity, so that the off-diagonal derivatives within that block vanish. The Codazzi identity then gives $\eta_i=0$ for $1\le i\le m$ and $z_i=0$ for $2\le i\le m$. The second order inequality at the point of contact retains the eigenvalue variation terms with positive gaps $\lambda_1-\lambda_i$, $i>m$, while all the omitted terms with zero gaps have vanishing numerators. Applying the preceding estimates with these contact conditions, we obtain the same inequalities for upper test functions. Alternatively, one may first apply the usual perturbation argument which makes the largest eigenvalue simple, and then pass to the limit.
\end{proof}

\subsection{Dirichlet Green's functions}
Let $D$ be a relatively compact smooth domain contained in $\{\tr_\omega\chi_u>R_0\}$. We denote by $G_D(x,y)$ the nonnegative Dirichlet Green's function of $-\Ln$, with integration taken with respect to $dV_\omega$, so that
\[
p(x)=\int_DG_D(x,y)q(y)\,dV_\omega(y)
\]
is the solution of $-\Ln p=q$ in $D$ with zero boundary values. We shall not need the symmetry of $G_D$.

In particular, let $\tau_D$ be the solution of $-\Ln\tau_D=1$ in $D$ with $\tau_D=0$ on $\partial D$. The barrier \eqref{eq:highbarrier} and the maximum principle imply
\begin{equation}\label{eq:torsionbound}
0\le\tau_D(x)=\int_DG_D(x,y)\,dV_\omega(y)
\le\frac{w_u(x)+C_{\mathrm{shift}}}{\kappa}
\le\frac{C_{\mathrm{shift}}}{\kappa}.
\end{equation}
This bound will be combined with a localized version of the Alexandrov estimate.

\begin{lemma}\label{lem:greenlocalized}
Let $\Ln$ be a smooth elliptic operator of complex Hessian type, whose positive coefficient matrix has $\omega$-trace one. Assume that
\[
\sup_{x\in D}\int_DG_D(x,y)\,dV_\omega(y)\le T_D,
\]
and that the coefficient matrix is bounded below by $\vartheta\omega^{-1}$ on a measurable set $E\subset D$. Then
\begin{equation}\label{eq:greenlocalized}
\sup_{x\in D}\int_EG_D(x,y)\,dV_\omega(y)
\le C(X,\omega,n,\vartheta)(1+T_D)
\Vol_\omega(E)^{1/(2n)}.
\end{equation}
In particular, no uniform lower bound on the ellipticity is needed outside $E$.
\end{lemma}
\begin{proof}
Consider the Green potential
\[
v_E(x)=\int_EG_D(x,y)\,dV_\omega(y).
\]
If it vanishes identically, there is nothing to prove. Otherwise, we set
\[
\varepsilon=\frac{\|v_E\|_{L^\infty(D)}}{2(1+T_D)}.
\]
Since $0\le\tau_D\le T_D$, we have
\[
\sup_D(v_E-\varepsilon\tau_D)
\ge\tfrac12\|v_E\|_{L^\infty(D)}>0.
\]
The maximum is therefore attained at an interior point $x_0$. Let $B_r(x_0)$ be a holomorphic coordinate ball of uniformly fixed radius, in one of a fixed finite collection of charts. Since the coefficient matrix has trace one, $\Ln|z-z(x_0)|^2$ is bounded from above, and we can choose $a>0$, depending only on the charts and the metric, such that
\[
a\Ln|z-z(x_0)|^2\le\tfrac12,
\qquad ar^2\le\tfrac12.
\]
On $D\cap B_r(x_0)$, we define
\[
\varphi=v_E-\varepsilon\tau_D
-\sup_D(v_E-\varepsilon\tau_D)
+a\varepsilon(r^2-|z-z(x_0)|^2).
\]
This function is nonpositive on the boundary of its domain. On the spherical part of the boundary, this follows from the definition of the supremum, while on the part lying in $\partial D$, it follows from the zero boundary values and the preceding lower bound for the supremum. At $x_0$, we have $\varphi(x_0)=a\varepsilon r^2$, while
\[
\Ln\varphi\ge-\one_E+\varepsilon/2\ge-\one_E.
\]
The Alexandrov estimate for real operators in dimension $2n$ \cite[Chapter 9]{GT} then gives
\[
a\varepsilon r^2\le\sup\varphi
\le C\Vol_\omega(E\cap B_r(x_0))^{1/(2n)}.
\]
The constant in this estimate is uniform. Indeed, the negative part of $\Ln\varphi$ is supported in $E$, and only there does the estimate involve the reciprocal of the determinant of the real coefficient matrix, which is controlled by the assumed lower bound on the coefficient matrix. Substituting the definition of $\varepsilon$, we obtain \eqref{eq:greenlocalized}.

The Alexandrov estimate applies equally to weak solutions of the Dirichlet problem with measurable right-hand side, as can also be seen by approximation. The assertion for the open sets used below then follows by a smooth exhaustion.
\end{proof}

Next, we estimate the volume of the sets where the trace is large; this step is valid for Hermitian metrics as well. Let $d\mu$ be the invariant measure of Lemma~\ref{lem:L1compact}, and choose $C_\mu$ with $dV_\omega\le C_\mu\,d\mu$. Since
\[
\tr_\omega\chi_u=\tr_\omega\chi+\Delta_\omega u>0,
\]
we have
\begin{equation}\label{eq:mass}
\begin{gathered}
\int_X\tr_\omega\chi_u\,d\mu
=\int_X\tr_\omega\chi\,d\mu,\\
\Vol_\omega\{\tr_\omega\chi_u>R\}
\le\frac{C_\mu}{R}\int_X\tr_\omega\chi\,d\mu.
\end{gathered}
\end{equation}
The integral on the right-hand side is positive as soon as the admissible family is nonempty. In the K\"ahler case, we may take $d\mu=dV_\omega$. The closedness of $\chi$ is not used in this calculation.

For $D\subset\{\tr_\omega\chi_u>R\}$ and
$E\subset D\cap\{\lambda_n<-\epsilon\tr_\omega\chi_u\}$, it follows from the ellipticity assumption \eqref{eq:badcone}, the bound \eqref{eq:torsionbound}, and Lemma~\ref{lem:greenlocalized} that
\begin{equation}\label{eq:greenvolume}
\sup_{x\in D}\int_EG_D(x,y)\,dV_\omega(y)
\le C_\epsilon R^{-1/(2n)}.
\end{equation}
Here and below, $R$ is assumed sufficiently large, depending on the fixed data and $\epsilon$.

\subsection{The localized gradient integral}
The main application of the Green's function estimates is the following consequence of uniform compactness, which controls an integral of the gradient without any pointwise bound on the gradient.

\begin{lemma}\label{lem:greenenergy}
For each fixed $\epsilon>0$, we have
\begin{equation}\label{eq:greenenergy}
\sup_{x\in D}\int_EG_D(x,y)|\partial w_u|_{\Ln}^2(y)\,dV_\omega(y)
\longrightarrow0\qquad(R\to\infty),
\end{equation}
uniformly over the family of solutions and over all
\[
D\subset\{\tr_\omega\chi_u>R\},\qquad
E\subset D\cap\{\lambda_n<-\epsilon\tr_\omega\chi_u\}.
\]
\end{lemma}
\begin{proof}
Let $v$ be a smooth admissible function with $\|u-v\|_{L^\infty}\le\eta$. By concavity,
\[
\Ln(u-v)\le\frac{h-F(A_v)}{\mathcal T}
\le H_0\varepsilon(R)\quad\hbox{on }D.
\]
Since $-2\eta\le u-v-\eta\le0$, it follows that
\[
\Ln\bigl((u-v-\eta)^2\bigr)
\ge-4\eta H_0\varepsilon(R)
+2|\partial(u-v)|_{\Ln}^2.
\]
The Green representation formula allows us to compare $(u-v-\eta)^2$ with its $\Ln$-harmonic extension from the boundary. Since both functions take values between $0$ and $4\eta^2$, we obtain
\[
\int_DG_D(x,y)\Ln\bigl((u-v-\eta)^2\bigr)(y)\,dV_\omega(y)
\le4\eta^2.
\]
Combining these inequalities with \eqref{eq:torsionbound}, and then using
$\partial w_u=\partial(u-v)+\partial(v-\ub)$, we find
\begin{equation}\label{eq:energysplit}
\begin{aligned}
\sup_{x\in D}\int_EG_D(x,y)|\partial w_u|_{\Ln}^2(y)\,dV_\omega(y)
\le{}&4\eta^2+
\frac{4\eta H_0C_{\mathrm{shift}}}{\kappa}\varepsilon(R)\\
&+C_\epsilon\|\nabla(v-\ub)\|_{L^\infty}^2R^{-1/(2n)}.
\end{aligned}
\end{equation}
For the last term, we have used \eqref{eq:greenvolume} and the trace normalization, which bounds $|\partial(v-\ub)|_{\Ln}^2$ by $|\nabla(v-\ub)|_\omega^2$.

Suppose now that the lemma is false. Then there exist $c_*>0$, numbers $R_j\to\infty$, solutions $u_j$, and sets $D_j,E_j$ for which the left-hand side of \eqref{eq:energysplit} is at least $c_*$. By Lemma~\ref{lem:C0compact}, we may pass to a uniformly Cauchy subsequence. We choose $\eta>0$ with $4\eta^2<c_*/2$, and then fix an index $N$ such that
\[
\|u_j-u_N\|_{L^\infty}\le\eta\qquad(j\ge N).
\]
We apply \eqref{eq:energysplit} with $v=u_N$, whose gradient is fixed, while $\varepsilon(R_j)$ and $R_j^{-1/(2n)}$ tend to zero. Letting $j\to\infty$, we obtain $c_*\le4\eta^2$, which is a contradiction. Note that the argument is applied to the Green's function of each solution separately, and does not require any convergence of the operators or of their kernels.
\end{proof}

\subsection{Proof of the Hessian estimate}\label{sec:greenmaximum}
\begin{proof}[Proof of Theorem~\ref{thm:kahler}]
Let $C\ge1$ be a constant dominating the constants in Lemma~\ref{lem:eigenvalue}, where the term $C/\lambda_1$ has been absorbed into $C$ on $\{\lambda_1\ge1\}$. We choose
\begin{equation}\label{eq:alphachoice}
0<\alpha\le\min\left\{\tfrac12,
\frac{\kappa}{32CC_{\mathrm{shift}}}\right\},
\qquad a_0=\frac{4C\alpha}{\kappa},
\end{equation}
and use the refined differential inequality with $\beta=\alpha$. We write $\delta=\delta_\alpha$ for the corresponding cone parameter.

Let $M=\sup_X\lambda_1$, and assume that $M$ is large. We work on the set
\[
D=\left\{x\in X:\left(\frac{\lambda_1(x)}{M}\right)^\alpha>\frac14\right\}
\]
with the test function
\[
\Phi=\left(\frac{\lambda_1}{M}\right)^\alpha-a_0w_u.
\]
By \eqref{eq:tracecone}, the trace tends to infinity uniformly on $D$ as $M\to\infty$. Moreover, the set $D$ is proper for large $M$, since at a maximum point of $u-\ub$, we have $A_u\le A_{\ub}$, so that $\lambda_1$ is bounded there in terms of the fixed subsolution.

Let $\rho\in[0,1]$ be a continuous cutoff function of $-\lambda_n/\lambda_1$, which vanishes when this ratio is at most $\delta/2$ and is equal to one when it is at least $\delta$. We use the refined inequality where $\rho<1$, and the global inequality elsewhere. The identity
\[
\Ln\left(\frac{\lambda_1}{M}\right)^\alpha
=\alpha\left(\frac{\lambda_1}{M}\right)^\alpha\Ln\log\lambda_1
+\frac{\left|\partial(\lambda_1/M)^\alpha\right|_{\Ln}^2}
{(\lambda_1/M)^\alpha}
\]
then gives
\begin{equation}\label{eq:Lt}
\Ln\left(\frac{\lambda_1}{M}\right)^\alpha
\ge-C\alpha
-\rho\frac{1-\alpha}{\alpha}
\frac{\left|\partial(\lambda_1/M)^\alpha\right|_{\Ln}^2}
{(\lambda_1/M)^\alpha}.
\end{equation}
Since $(\lambda_1/M)^\alpha>1/4$ on $D$ and
$\partial(\lambda_1/M)^\alpha=\partial\Phi+a_0\partial w_u$, it follows from \eqref{eq:highbarrier} and \eqref{eq:alphachoice} that
\[
\Ln\Phi\ge C\alpha
-\frac8\alpha\rho|\partial\Phi|_{\Ln}^2
-\frac{8a_0^2}{\alpha}\rho|\partial w_u|_{\Ln}^2.
\]
To absorb the term involving $\partial\Phi$, we take an exponential:
\begin{equation}\label{eq:LZ}
\Ln e^{8\Phi/\alpha}
\ge-\frac{64a_0^2}{\alpha^2}
\rho|\partial w_u|_{\Ln}^2e^{8\Phi/\alpha}.
\end{equation}
The remaining coefficient is supported in the region where
\[
\lambda_n<-\frac{\delta}{2}\lambda_1
\le-\frac{\delta}{2n}\tr_\omega\chi_u.
\]
Therefore, by Lemma~\ref{lem:greenenergy}, we have, for $M$ sufficiently large,
\begin{equation}\label{eq:smallpotential}
\frac{64a_0^2}{\alpha^2}
\sup_{x\in D}\int_DG_D(x,y)
\rho(y)|\partial w_u|_{\Ln}^2(y)\,dV_\omega(y)
\le\tfrac12.
\end{equation}

In order to apply the maximum principle, we subtract the Green potential of the right-hand side of \eqref{eq:LZ}, which is given by
\[
\Psi(x)=\frac{64a_0^2}{\alpha^2}
\int_DG_D(x,y)\rho(y)|\partial w_u|_{\Ln}^2(y)
e^{8\Phi(y)/\alpha}\,dV_\omega(y).
\]
Then $\Psi$ has zero boundary values, and $\Ln(e^{8\Phi/\alpha}-\Psi)\ge0$. Hence, in view of \eqref{eq:smallpotential},
\[
\sup_De^{8\Phi/\alpha}
\le\sup_{\partial D}e^{8\Phi/\alpha}
+\frac12\sup_De^{8\Phi/\alpha},
\]
and therefore
\begin{equation}\label{eq:Zmax}
\sup_De^{8\Phi/\alpha}
\le2\sup_{\partial D}e^{8\Phi/\alpha}.
\end{equation}
If the boundary is not smooth, we apply the argument on a smooth exhaustion of $D$ and pass to the limit; the estimates are uniform, and by continuity we obtain the boundary supremum appearing above. At points where the largest eigenvalue is multiple, we use the formulation of Lemma~\ref{lem:eigenvalue} in terms of upper test functions, together with the maximum principle for the fixed smooth operator.

At a point where $\lambda_1=M$, we have $\Phi=1-a_0w_u\ge1$. On the other hand, on $\partial D$, the choice of $\alpha$ gives $a_0C_{\mathrm{shift}}\le1/8$, and hence
\[
\Phi=\tfrac14-a_0w_u
\le\tfrac14+a_0C_{\mathrm{shift}}\le\tfrac12.
\]
Thus \eqref{eq:Zmax} would imply $e^{8/\alpha}\le2e^{4/\alpha}$, which is impossible for $\alpha\le1/2$. It follows that $M$ is uniformly bounded, and the inequalities \eqref{eq:tracecone} then bound all the eigenvalues, which proves \eqref{eq:kahler}.

Finally, because of the positive lower bound for $h$, these bounded eigenvalues remain in a compact subset of $\Gamma$. The stated $C^{2,\alpha_0}$ bound then follows from uniform ellipticity, the complex Evans-Krylov estimate, and the Schauder estimates for $\Delta_\omega$, exactly as in the proof of Corollary~\ref{cor:hermitian}.
\end{proof}

\end{document}